\documentclass[11pt,a4paper]{article}
\usepackage[T1]{fontenc}
\usepackage[utf8]{inputenc}
\usepackage{eurosym}
\usepackage{amsfonts}
\usepackage[margin=1.1in]{geometry}
\usepackage{amsmath,amsthm,amssymb,mathrsfs}
\usepackage{graphicx}
\usepackage{booktabs}
\usepackage{caption}
\usepackage{subcaption}
\usepackage{xcolor}
\usepackage[colorlinks=true, linkcolor=RoyalBlue, citecolor=RoyalBlue, urlcolor=RoyalBlue]{hyperref}
\usepackage{enumitem}
\usepackage{listings}
\usepackage{algorithm}
\usepackage{algorithmic}
\usepackage{float}
\usepackage[section]{placeins}
\usepackage{authblk}
\usepackage{microtype}
\usepackage{titlesec}
\usepackage{indentfirst}
\usepackage[nameinlink,capitalize]{cleveref}
\usepackage{mathpazo}
\usepackage[scaled=0.95]{helvet}
\usepackage{courier}

\definecolor{RoyalBlue}{RGB}{0, 35, 102}
\titleformat{\section}{\large\bfseries\scshape}{\thesection.}{0.5em}{}
\titleformat{\subsection}{\normalsize\bfseries}{\thesubsection.}{0.5em}{}
\theoremstyle{plain}
\newtheorem{theorem}{Theorem}[section]

\theoremstyle{definition}
\newtheorem{definition}[theorem]{Definition}

\theoremstyle{remark}
\newtheorem{remark}[theorem]{Remark}
\numberwithin{equation}{section}

\begin{document}

\title{{\LARGE \textbf{A Bernoulli-Type Localization Principle for Nonlinear
Diffusion Equations: Structural Description of Solutions and First
Application to Image Processing}}}
\author{\textbf{Dragos-Patru Covei}}
\date{{\small \today}}
\maketitle

\begin{abstract}
\noindent This paper introduces a genuinely new approach to the analysis of
nonlinear diffusion equations of porous medium type with time-dependent
growth rates: the \emph{localization of solutions} through an associated
Bernoulli-type ordinary differential equation. This principle---which we
term the \emph{Bernoulli localization principle}---provides, for the first
time, a complete structural description of solutions to the
initial-boundary-value problem by explicitly decoupling the temporal
dynamics, governed by a Bernoulli-type ODE whose solution is available in
closed form, from the spatial profile, determined by a time-independent
sublinear elliptic equation of Brezis--Oswald type. The resulting
decomposition changes the way this class of partial differential equations
has been viewed: rather than analyzing the full parabolic PDE by purely
qualitative methods, one solves a scalar ODE and a stationary elliptic
problem, and their combination yields a complete, explicit structural
description of every solution. We establish three main results: (i)
existence and uniqueness of separable classical solutions via the
transformation of the parabolic problem into an elliptic framework coupled
with the Bernoulli ODE (Theorem~\ref{th1}); (ii) a comparison principle
showing that the initial ordering of data and growth rates is preserved
throughout evolution (Theorem~\ref{th-comp}); and (iii) a general existence
and uniqueness result for non-separable solutions via monotone iteration
based on the subsolution--supersolution technique, with explicit
localization bounds constructed from the Bernoulli ODE
(Theorem~\ref{th2}). To demonstrate the applied character of this
framework, we present what is, to our knowledge, the \emph{first
application} of a porous-medium-type diffusion equation to digital image
processing, specifically image denoising, where the model achieves results
that match or exceed those of the established Perona--Malik anisotropic
diffusion filter. A complete and transparent description of the Python
implementation is provided. The framework is also relevant to other
applied domains---physics (thermal diffusion in porous media), biology
(population dynamics, tumor growth), economics (diffusion of innovations,
wealth distribution), and engineering (groundwater flow, materials
processing)---via concrete models and verifiable predictions.
\end{abstract}


\affil{\small Department of Applied Mathematics, The Bucharest University of Economic Studies\\
Piata Romana, 1st district, Postal Code: 010374, Postal Office: 22, Romania\\
\texttt{dragos.covei@csie.ase.ro}}

\vspace{1em}

\noindent \textbf{Keywords:} \textit{Nonlinear diffusion; Bernoulli
localization principle; Structural description of solutions; Image
denoising; Porous medium equation; Gurtin--MacCamy model; Existence and
uniqueness; Monotone iteration; Perona--Malik comparison; Cross-disciplinary
applications.}

\vspace{0.5em}

\noindent \textbf{MSC 2020:} \textit{35K55; 35K57; 35J20; 35B40; 35D30;
92D25; 94A08.}

\vspace{2em}


\section{Introduction}

\subsection{Motivation and the classical viewpoint}

The theory of nonlinear diffusion equations has been a central pillar of
mathematical analysis and applied mathematics for over half a century.
Equations of the form $\partial _{t}v=\Delta v^{\alpha }$ with $\alpha >1$%
---the so-called porous medium equations---arise naturally in contexts
ranging from the flow of gases through porous rock formations to the
dispersal of biological populations in heterogeneous environments \cite%
{V, GM, Bar}. When augmented by a reaction or growth term, these equations
become indispensable in modeling processes where diffusion and proliferation
interact: thermal conduction in materials with temperature-dependent
conductivity \cite{DiBenedetto}, density-dependent dispersal of biological
species \cite{GM, Murray, Okubo}, groundwater flow through soils with
nonlinear permeability \cite{Bear}, and even the propagation of innovations
through social networks \cite{Rogers}.

Despite decades of intensive study, the methods available for understanding
the \emph{structure} of solutions to these equations have remained largely
indirect. Comparison principles, energy estimates, regularity theory, and
asymptotic analysis provide existence, uniqueness, and qualitative
information, but they typically do not furnish an \emph{explicit structural
description} of the solution. The celebrated work of Brezis and Oswald \cite%
{BO} established the existence and uniqueness of solutions to sublinear
elliptic problems, while Brezis and Kamin \cite{BK} extended these results
to unbounded domains; the monograph of V\'{a}zquez \cite{V} gave a
comprehensive treatment of existence, uniqueness, regularity, and
asymptotics for the porous medium equation; and the foundational framework of
Gurtin and MacCamy \cite{GM} provided the biological underpinning of
density-dependent diffusion. In all these approaches, however, the solution
is characterized \emph{implicitly}: it exists, it is unique, it has certain
qualitative properties---but its internal structure remains opaque.

\subsection{A new idea: the Bernoulli localization principle}

In this paper we introduce a genuinely new approach that fundamentally
changes this picture. The central observation is disarmingly simple yet far-%
reaching: the solutions of the nonlinear parabolic problem can be \emph{%
localized}---that is, explicitly described, bounded, and structurally
characterized---by means of an associated Bernoulli-type ordinary
differential equation. This localization principle rests on a decomposition
of the solution into a temporal factor $S(t)$, governed by a scalar
Bernoulli ODE with a closed-form solution, and a spatial profile $%
u^{1/\alpha }(x)$, determined by a stationary elliptic equation. The
combination of these two factors provides a \emph{complete structural
description} of the solution that is unprecedented for this class of
equations.

To be precise, we study the nonlinear diffusion equation%
\begin{equation}
\frac{\partial v}{\partial t}(x,t)-\Delta v^{\alpha }(x,t)=\mu \left(
t\right) v(x,t),\quad \text{in }\Omega \times (0,T),  \label{nd}
\end{equation}%
where $\alpha >1$, $T>0$ (possibly infinite), $\mu :\left[ 0,\infty \right)
\rightarrow \left( 0,\infty \right) $ is a continuous time-dependent
growth rate (additionally bounded whenever $T$ is infinite), and $%
\Omega \subset \mathbb{R}^{N}$ ($N\geq 1$) is a bounded smooth domain with
hostile boundaries ($\partial \Omega $). Problem \eqref{nd} is supplemented
with boundary and initial conditions:%
\begin{align}
v(x,t)& =0,\quad \text{on }\partial \Omega \times (0,T),  \label{bnd1} \\
v(x,0)& =v_{0}(x),\quad \text{on }\Omega .  \label{bnd2}
\end{align}

Here, $v(x,t)$ represents the evolving population density (or, analogously,
temperature, concentration, or image intensity), $\Delta v^{\alpha }(x,t)$
embodies a power-law diffusion process, $\alpha $ governs the nonlinearity
of diffusion, $\mu (t)$ represents the time-dependent growth or decay rate, $%
T$ defines the observation time frame, the boundary condition \eqref{bnd1}
ensures the solution vanishes at the boundary of the domain, and the initial
condition \eqref{bnd2} specifies the initial state.

The key methodological insight of our study is that a
separation-of-variables ansatz, combined with the identification of the
temporal factor as the solution of a Bernoulli-type ODE, yields not merely
an existence result but an \emph{explicit structural formula} for the
solution. Specifically, we show that the problem \eqref{nd}--\eqref{bnd1}--%
\eqref{bnd2} admits solutions of the form%
\begin{equation}
v_{\gamma }\left( x,t\right) =S\left( t\right) u^{\frac{1}{\alpha }}\left(
x\right) ,\quad S\left( t\right) =\left[ e^{\left( 1-\alpha \right)
\int_{0}^{t}\mu \left( s\right) ds}\left( \gamma ^{1-\alpha }+(\alpha
-1)\int_{0}^{t}e^{-\left( 1-\alpha \right) \int_{0}^{\tau }\mu \left(
s\right) ds}d\tau \right) \right] ^{\frac{1}{1-\alpha }},  \label{2}
\end{equation}%
where $u\left( x\right) $ is the unique positive solution of the associated
elliptic problem%
\begin{equation}
\left\{ 
\begin{array}{ccc}
-\Delta u\left( x\right) =u^{\frac{1}{\alpha }}\left( x\right) & in & \Omega
, \\ 
u\left( x\right) >0 & in & \Omega , \\ 
u\left( x\right) =0 & on & \partial \Omega ,%
\end{array}%
\right.  \label{dir}
\end{equation}%
and the initial condition is satisfied as%
\begin{equation}
v_{0}(x)=v_{\gamma }\left( x,0\right) =S\left( 0\right) u^{\frac{1}{\alpha }%
}\left( x\right) ,  \label{bc0}
\end{equation}%
with $S\left( 0\right) =\gamma \in \left( 0,\infty \right) $ given.

This representation is not merely a calculational convenience; it reveals
the \emph{structural anatomy} of the solution. The temporal dynamics are
entirely governed by a scalar Bernoulli-type ODE, while the spatial profile
is determined once and for all by a stationary sublinear elliptic problem.
This decomposition has several profound consequences: (a)~the long-time
behavior is immediately readable from the asymptotics of $S(t)$;
(b)~parameter sensitivity is transparent, since $\alpha $ affects both $S$
and $u$ in explicitly traceable ways; and (c)~numerical implementation
reduces to solving an elliptic equation (a mature, well-understood task)
plus evaluating the closed-form expression for $S(t)$.

\subsection{Literature and context}

It is well known that problem \eqref{nd}, together with the boundary
conditions \eqref{bnd1}--\eqref{bnd2}, is typically analyzed within the
theoretical framework introduced by Gurtin and MacCamy \cite{GM},
which provides a foundational approach to diffusion processes with
density-dependent dispersal. Subsequent research has extended this setting
in various directions; for instance, when $\mu \left( t\right) $ is a
constant and $\Omega =\mathbb{R}^{N}$, the problem has been studied in
detail by He \cite{H}. When $\mu \left( t\right) =0$, the equation %
\eqref{nd} reduces to the classical porous medium equation,%
\begin{equation*}
\frac{\partial v}{\partial t}(x,t)=\Delta v^{\alpha }(x,t),\quad \text{in }%
\Omega \times (0,T),
\end{equation*}%
well known for its degenerate parabolic structure, finite-speed propagation, and rich regularity theory; see \cite{K,V,V2}
for detailed treatments.

In biological systems, nonlinear diffusion equations capture the complex
dynamics of microbial populations where growth and dispersal interact in
nontrivial ways \cite{GM, Okubo, Murray}. Similar models describe thermal
processes in non-homogeneous materials, phase transitions, and the evolution
of microstructural defects \cite{V, DiBenedetto}. The foundational work of
Gurtin and MacCamy \cite{GM} introduced the density-dependent diffusion
model for biological populations, where the diffusion coefficient depends on
the population density itself, leading to equations of the form $%
v_{t}=\Delta v^{\alpha }$ with $\alpha >1$ (slow diffusion). The monograph
by V\'{a}zquez \cite{V} provides a comprehensive treatment of the
mathematical theory, while the work of Aronson and Peletier \cite{AP}
established fundamental results on the large-time behavior of solutions in
bounded domains.

The subsolution--supersolution method, systematically developed by Pao \cite%
{PAO} for parabolic systems, provides a powerful framework for establishing
existence and uniqueness in the general non-separable case. Recent progress
on population models with nonlinear survival structures has been reported in 
\cite{Covei1} and \cite{Covei2}. These contributions, while significant,
characterize solutions \emph{qualitatively}. What has been missing---and
what this paper provides---is a principle that makes the solution structure 
\emph{explicit}.

\subsection{Broader applicability}

The broader applicability of the Bernoulli localization framework extends
well beyond biological dynamics and encompasses several scientific and
engineering disciplines.

In \textbf{physics}, equations of the form \eqref{nd} describe thermal
diffusion in porous media where the thermal conductivity is a power law of
the temperature \cite{Bar}. The localization principle predicts explicit
upper and lower bounds on the temperature field at any time, given the
initial temperature profile.

In \textbf{biology}, density-dependent dispersal models such as \eqref{nd}
arise naturally in ecology \cite{Okubo, Murray, Fisher}, where $v(x,t)$
represents the density of a species and $\mu (t)$ captures time-varying
birth-death dynamics. The structural decomposition allows explicit
computation of maximum and minimum population densities at any time $t$.

In \textbf{economics and social sciences}, nonlinear diffusion models
capture the spread of innovations through a population \cite{Rogers}, where $%
v(x,t)$ represents adoption density and the power-law diffusion reflects
peer-influence effects. The localization principle yields verifiable bounds
on market penetration as a function of time.

In \textbf{engineering}, the porous medium equation governs groundwater flow 
\cite{Bear}, oil recovery from reservoirs, and materials processing (e.g.,
sintering). In each of these contexts, the Bernoulli localization principle
provides quantitative predictions that can be tested against field data or
laboratory measurements.

We elaborate on these cross-disciplinary connections in
Section~\ref{cross}.

\subsection{Main contributions and novelty}

In summary, this paper establishes three key theoretical results
complemented by a first application to image processing:

\begin{itemize}
\item \textbf{Theorem \ref{th1}} (Bernoulli localization for separable
solutions)\textbf{:} An existence and uniqueness result for separable
solutions, obtained by transforming the parabolic problem into an elliptic
equation coupled with a Bernoulli-type ordinary differential equation. The
solution is given in closed form by \eqref{2}.

\item \textbf{Theorem \ref{th-comp}} (Comparison principle)\textbf{:} A
comparison result with respect to the initial conditions and growth rates,
showing that the population distribution preserves its initial ordering
throughout evolution.

\item \textbf{Theorem \ref{th2}} (Localization bounds for general
solutions)\textbf{:} A general existence and uniqueness result for non-%
separable solutions, established via the monotone iteration method based on
subsolution--supersolution techniques, where the sub- and supersolutions are
explicitly constructed from the Bernoulli ODE, thereby providing concrete
localization bounds.

\item \textbf{Image denoising application:} The first use of a porous-medium-%
type diffusion equation for image denoising, with a transparent
line-by-line explanation of the Python implementation, demonstrating
performance comparable to or exceeding the classical Perona--Malik filter.
\end{itemize}

\noindent The novelty of this work is threefold. First, the Bernoulli
localization principle itself is new: it provides a structural
decomposition that has not been available for this class of equations.
Second, the explicit localization bounds in Theorem~\ref{th2}---constructed
from the temporal factor $S(t)$ and the elliptic sub- and
supersolutions---go beyond classical qualitative comparison principles by
furnishing quantitatively sharp barriers. Third, the application to image
processing is new: the Gurtin--MacCamy diffusion model has not previously
been used for image denoising, and our results show that it outperforms the
classical Perona--Malik filter, especially at higher noise levels.

\subsection{Statement of main results}

We now present the main results in their abstract formulation.

\begin{theorem}
\label{th1}Let $\Omega \subset \mathbb{R}^{N}$ be a bounded,
arcwise-connected open set with a compact closure and boundary $\partial
\Omega$ of class $C^{3}$. Under these conditions, the nonlinear diffusion
problem \eqref{nd}--\eqref{bnd1} admits a unique solution $v_{\gamma}(x,t)$,
given by \eqref{2}, where $u\in C^{2}(\Omega) \cap C^{1}(\overline{\Omega})$
exists uniquely, and the initial boundary condition \eqref{bnd2} is given by %
\eqref{bc0}.
\end{theorem}

Our methodology for proving Theorem \ref{th1} departs from classical
strategies, such as those found in \cite{BK, GM, V2}, by harnessing a
transformation that simplifies the analysis even under challenging regimes.
This approach not only strengthens the theoretical framework but also paves
the way for robust numerical implementations.

\begin{theorem}
\label{th-comp} Let $v_{0},w_{0}\in C(\overline{\Omega })$ be two initial
conditions such that $v_{0}(x)\leq w_{0}(x)$ for all $x\in \overline{\Omega }
$. Assume that $\mu _{1}\left( t\right) $ and $\mu _{2}\left( t\right) $ are
continuous and bounded functions such that $\mu _{1}\left( t\right) \leq \mu
_{2}\left( t\right) $. If $v(x,t)$ and $w(x,t)$ are the unique classical
solutions to problem \eqref{nd}--\eqref{bnd2} corresponding to $(v_{0},\mu
_{1})$ and $(w_{0},\mu _{2})$, respectively, then 
\begin{equation}
v(x,t)\leq w(x,t)\quad \text{for all }\left( x,t\right) \in \overline{\Omega 
}\times \lbrack 0,T).
\end{equation}%
Intuitively, this result shows that if the initial population $v_{0}$ is
smaller than $w_{0}$ and the growth rate $\mu _{1}(t)$ is smaller than $\mu
_{2}(t)$, then solution $v$ stays below $w$ during the entire observation
period.
\end{theorem}

Theorem \ref{th1} and Theorem \ref{th-comp} lead to a new result that is
presented here for the first time.

\begin{theorem}
\label{th2}Let $\Omega \subset \mathbb{R}^{N}$ be a bounded,
arcwise-connected open set with a compact closure and boundary $\partial
\Omega $ of class $C^{3}$. Assume that $u_{-},u_{+}\in C^{2}(\Omega )\cap C(%
\overline{\Omega })$ satisfy 
\begin{align}
-\Delta u_{-}& \leq u_{-}^{1/\alpha }\quad \text{in }\Omega , & u_{-}>0\text{
in }\Omega ,& u_{-}=0\text{ on }\partial \Omega ,  \label{ell-sub} \\
-\Delta u_{+}& \geq u_{+}^{1/\alpha }\quad \text{in }\Omega , & u_{+}>0\text{
in }\Omega ,& u_{+}=0\text{ on }\partial \Omega ,  \label{ell-super}
\end{align}%
and $u_{-}\leq u_{+}$ in $\overline{\Omega }$. Fix $\gamma >0$ and let $S$
be the unique solution of 
\begin{equation}
S^{\prime }(t)+S^{\alpha }(t)\,=\mu \left( t\right) S\left( t\right) ,\qquad
S(0)=\gamma ,  \label{S-ODE}
\end{equation}%
as given by Equation \eqref{2}. We define $v_{-}(x,t)=S(t)u_{-}^{1/\alpha
}(x)$ and $v_{+}(x,t)=S(t)u_{+}^{1/\alpha }(x)$. If $v_{0}\in C(\overline{%
\Omega })$ satisfies $v_{-}(x,0)\leq v_{0}(x)\leq v_{+}(x,0)$ in $\overline{%
\Omega }$, then there exists a unique classical solution $v(x,t)$ to %
\eqref{nd}-\eqref{bnd2}, such that $v_{-}(x,t)\leq v(x,t)\leq v_{+}(x,t)$
for all $\left( x,t\right) \in \overline{\Omega }\times \left[ 0,T\right) $.
\end{theorem}

A careful examination of Theorems \ref{th1} and \ref{th2} and their proofs
shows that both results play an essential role, even though Theorem \ref{th1}
is a particular case of Theorem \ref{th2}. Theorem \ref{th1} remains
significant because its proof lends itself to straightforward implementation
in a programming language through simplified numerical procedures, whereas
Theorem \ref{th2} applies to a more general situation in which the solution
is not necessarily separable. As an illustration, we present a well-known
result that highlights the distinct roles played by Theorems~\ref{th1} and %
\ref{th2}.

\begin{remark}
\label{rem} There always exists a continuous function 
\begin{equation*}
u_{-}^{+}:\overline{\Omega }\rightarrow \lbrack 0,\infty )
\end{equation*}%
such that 
\begin{equation}
u_{-}(x)\leq u_{-}^{+}(x)\leq u_{+}(x)\quad \text{for all }x\in \overline{%
\Omega }.  \label{arit}
\end{equation}%
For instance, for any continuous function $\theta :\overline{\Omega }%
\rightarrow \lbrack 0,1]$, the function 
\begin{equation*}
u_{-}^{+}(x):=(1-\theta (x))\,u_{-}(x)+\theta (x)\,u_{+}(x)
\end{equation*}%
satisfies Eq. \eqref{arit}.
\end{remark}

Therefore, Theorem~\ref{th-comp} strengthens both the first (Theorem~\ref%
{th1}) and third result (Theorem~\ref{th2}). First, it plays a key role in
the proof of Theorem~\ref{th2}. Second, when combined with Theorem~\ref{th1}%
, it enables the construction of a subsolution and supersolution pair, which
in turn makes the implementation of the results in a computer language
feasible.

\subsection{Organization of the paper}

The remainder of this paper is organized as follows. Section~\ref{loc-sec}
presents the Bernoulli localization principle in detail, explaining why it
constitutes a new structural framework for this class of equations.
Section~\ref{thp1} provides complete proofs of Theorems \ref{th1}--\ref{th2}%
. In Section~\ref{ef}, we derive the model from conservation laws and energy
considerations, following the framework of Gurtin and MacCamy.
Section~\ref{cross} discusses cross-disciplinary relevance, presenting
concrete models and verifiable predictions for physics, biology, economics,
and engineering. Section~\ref{a4} provides the complete Python
implementation for image denoising and presents a detailed comparison with
the Perona--Malik filter, including a transparent explanation of every
component of the code. Finally, Section~\ref{cf} presents conclusions and
future research directions.


\section{The Bernoulli Localization Principle: A New Structural Framework %
\label{loc-sec}}

In this section we describe in detail the central conceptual contribution of
this paper: the Bernoulli localization principle. We explain what it asserts,
why it is new, and how it changes the way one can think about the solutions
of porous-medium-type diffusion equations with time-dependent growth.

\subsection{The principle}

\begin{definition}[Bernoulli localization]
\label{def-loc} We say that the solutions of the parabolic problem %
\eqref{nd}--\eqref{bnd1}--\eqref{bnd2} are \emph{localized by the
Bernoulli ODE} if there exist functions $v_{-}(x,t)$ and $v_{+}(x,t)$ of
the product form%
\begin{equation}
v_{\pm }(x,t)=S(t)\,u_{\pm }^{1/\alpha }(x),  \label{loc-form}
\end{equation}%
where $S(t)$ is the unique positive solution of the Bernoulli-type ODE%
\begin{equation}
S^{\prime }(t)+S^{\alpha }(t)=\mu (t)\,S(t),\qquad S(0)=\gamma >0,
\label{loc-ode}
\end{equation}%
and $u_{-}$, $u_{+}$ are, respectively, a subsolution and a supersolution
of the stationary elliptic problem \eqref{dir} with $u_{-}\leq u_{+}$ in $%
\overline{\Omega }$, such that every classical solution $v$ of %
\eqref{nd}--\eqref{bnd1}--\eqref{bnd2} with initial data satisfying $%
v_{-}(x,0)\leq v_{0}(x)\leq v_{+}(x,0)$ is trapped between these barriers:%
\begin{equation}
v_{-}(x,t)\leq v(x,t)\leq v_{+}(x,t)\quad \text{for all }(x,t)\in 
\overline{\Omega }\times \lbrack 0,T).  \label{loc-ineq}
\end{equation}
\end{definition}

The power of this definition lies in the explicit, factored form of the
barriers. The temporal factor $S(t)$ is \emph{universal}: it depends only on
$\alpha $, $\mu (\cdot )$, and the initial parameter $\gamma $, not on the
spatial domain $\Omega $ or the spatial profile $u$. It is given in closed
form by Equation~\eqref{2}. The spatial factors $u_{\pm }^{1/\alpha }$ are
determined by the stationary elliptic problem on $\Omega $---a problem for
which existence, uniqueness, and regularity are classical \cite{BO, G}.

\subsection{Why the principle is new}

The novelty of the Bernoulli localization principle can be appreciated by
contrasting it with the approaches that have dominated the literature.

\begin{enumerate}
\item \textbf{Classical comparison principles} \cite{PAO, V} establish that
if a subsolution lies below a supersolution at $t=0$ and on the parabolic
boundary, then the same ordering persists for all time. These principles are
extremely powerful but remain \emph{qualitative}: they do not specify the
functional form of the sub- and supersolutions, nor do they relate the
temporal and spatial dependences explicitly. The Bernoulli localization
principle, by contrast, produces barriers of a specific, analytically
tractable product form, thereby converting a qualitative ordering into a
quantitative structural description.

\item \textbf{Energy and variational methods} \cite{V, DiBenedetto}
establish existence via minimization of suitable energy functionals. These
methods yield weak solutions and, through regularity bootstrapping, classical
solutions. However, the solutions so obtained remain implicit; the energy
functional tells us that a minimizer exists, but not what it looks like. The
Bernoulli localization principle, in contrast, tells us precisely what the
solution looks like: it is the product of a temporal Bernoulli factor and a
spatial elliptic profile.

\item \textbf{Self-similar and asymptotic analyses} \cite{AP, V2} provide
information about the long-time behavior of solutions, typically showing
convergence to a self-similar profile or a steady state. The Bernoulli
localization principle subsumes these results: the long-time behavior of $%
S(t)$ is immediately readable from the explicit formula \eqref{2},
yielding the asymptotic profile without further analysis.
\end{enumerate}

\subsection{Structural consequences}

The Bernoulli localization principle has several immediate structural
consequences.

\textbf{Complete structural description.} Every solution of %
\eqref{nd}--\eqref{bnd1}--\eqref{bnd2} lying in the ordered interval $%
[v_{-},v_{+}]$ is sandwiched between two explicitly known functions. In the
separable case, the solution \emph{is} such a product, and the description
is exact rather than merely bounding.

\textbf{Temporal--spatial decoupling.} The full parabolic PDE is replaced by
two independent, simpler problems: a scalar Bernoulli ODE for the temporal
factor and a stationary elliptic PDE for the spatial profile. Each can be
analyzed and solved independently, and their combination yields the solution
of the original problem.

\textbf{Explicit parameter dependence.} The dependence of the solution on
the parameters $\alpha $, $\gamma $, $\mu (\cdot )$, and $\Omega $ is
transparently readable from the formula \eqref{2} and the elliptic problem %
\eqref{dir}. This facilitates sensitivity analysis, parameter fitting in
applications, and the design of numerical experiments.

\textbf{Algorithmic accessibility.} The localization principle translates
directly into an algorithm: (i) solve the elliptic problem \eqref{dir}
by a standard numerical method (finite elements, finite differences, or the
iterative scheme in Theorem~\ref{th1}); (ii) evaluate $S(t)$ from the
closed-form expression \eqref{2}; (iii) form the product $v(x,t)=S(t)\,%
u^{1/\alpha }(x)$. No parabolic solver is needed for the separable case.


\section{Proof of the Main Results \label{thp1}}

In the following proof of Theorem \ref{th1}, we apply the successive
approximation method, because our goal is to implement the result using a
computer language.

\begin{proof}[Proof of Theorem \protect\ref{th1}]
We first prove the existence of solution to the problem \eqref{dir}. Let $%
\phi _{1}\in C^{2}(\overline{\Omega })$ be the first positive eigenfunction
corresponding to the first eigenvalue $\lambda _{1}$ of the problem%
\begin{equation*}
\left\{ 
\begin{array}{lll}
-\Delta \phi (x)=\lambda \phi (x), & in & \Omega , \\ 
\phi (x)=0, & on & \partial \Omega ,%
\end{array}%
\right.
\end{equation*}%
and $w\in C^{2}(\overline{\Omega })$ be the unique solution of the problem 
\begin{equation*}
\left\{ 
\begin{array}{ccc}
-\Delta w(x)=1, & in & \Omega , \\ 
w(x)=0, & on & \partial \Omega .%
\end{array}%
\right.
\end{equation*}%
We choose%
\begin{equation*}
\sigma =\frac{1}{\lambda _{1}^{\frac{\alpha }{\alpha -1}}\max_{x\in 
\overline{\Omega }}\phi _{1}\left( x\right) }\text{ and }c=\max \left\{
\sigma \lambda _{1}\max_{x\in \overline{\Omega }}\phi _{1},\left( \max_{x\in 
\overline{\Omega }}w\left( x\right) \right) ^{\frac{1}{\alpha -1}}\right\} .
\end{equation*}%
Then, the function $\underline{u}(x)=\sigma \phi _{1}$ is a subsolution of (%
\ref{dir}) whereas $\overline{u}\left( x\right) =cw(x)$ is a supersolution.
Indeed, 
\begin{align*}
-\Delta \underline{u}(x)-\left( \underline{u}(x)\right) ^{\frac{1}{\alpha }%
}& =\sigma \lambda _{1}\phi _{1}\left( x\right) -\left( \sigma \phi
_{1}\right) ^{\frac{1}{\alpha }} \\
& =\sigma \lambda _{1}\phi _{1}\left( x\right) -\sigma ^{\frac{1}{\alpha }%
}\left( \phi _{1}\left( x\right) \right) ^{\frac{1}{\alpha }} \\
& =\sigma \left( \lambda _{1}\phi _{1}\left( x\right) \right) ^{\frac{1}{%
\alpha }}\left[ \left( \lambda _{1}\phi _{1}\left( x\right) \right) ^{\frac{%
\alpha -1}{\alpha }}-\frac{1}{\sigma }\left( \frac{\sigma }{\lambda _{1}}%
\right) ^{\frac{1}{\alpha }}\right] \\
& \leq \sigma \left( \lambda _{1}\phi _{1}\left( x\right) \right) ^{\frac{1}{%
\alpha }}\left[ \left( \lambda _{1}\max_{x\in \overline{\Omega }}\phi
_{1}\left( x\right) \right) ^{\frac{\alpha -1}{\alpha }}-\frac{1}{\sigma }%
\left( \frac{\sigma }{\lambda _{1}}\right) ^{\frac{1}{\alpha }}\right] =0
\end{align*}%
and%
\begin{align*}
-\Delta \overline{u}(x)-\left( \overline{u}\left( x\right) \right) ^{\frac{1%
}{\alpha }}& =c-\left( cw\left( x\right) \right) ^{\frac{1}{\alpha }}\geq
c-\left( c\max_{x\in \overline{\Omega }}w\left( x\right) \right) ^{\frac{1}{%
\alpha }} \\
& =c^{\frac{1}{\alpha }}\left[ c^{\frac{\alpha -1}{\alpha }}-\left(
\max_{x\in \overline{\Omega }}w\left( x\right) \right) ^{\frac{1}{\alpha }}%
\right] \geq 0.
\end{align*}%
Now, since%
\begin{equation*}
\left\{ 
\begin{array}{lll}
-\Delta \lbrack \overline{u}(x)-\underline{u}(x)]=c-\sigma \lambda _{1}\phi
_{1}\geq c-\sigma \lambda _{1}\max_{x\in \overline{\Omega }}\phi _{1}\geq 0
& in & \Omega \\ 
\overline{u}(x)-\underline{u}(x)=0 & on & \partial \Omega%
\end{array}%
\right.
\end{equation*}%
it follows from the maximum principle that $\underline{u}(x)\leq \overline{u}%
(x)$, for all $x\in \overline{\Omega }$.

Next, we define the sequence $\{u^{k}\}_{k\geq 0}$ as 
\begin{equation}
u^{0}(x)=u_{0}(x),\text{ }-\Delta u^{k}\left( x\right) =\left( u^{k-1}\left(
x\right) \right) ^{\frac{1}{\alpha }}\text{ for }k=1,2,...\text{ and }x\in
\Omega  \label{iter}
\end{equation}%
where%
\begin{equation*}
u_{0}(x)=\underline{u}(x)\quad \forall \,x\in \overline{\Omega },
\end{equation*}%
with $\underline{u}(x)$ being a subsolution of the problem 
\begin{equation}
-\Delta u(x)=u(x)^{\frac{1}{\alpha }},\quad u(x)>0\text{ in }\Omega ,\qquad
u(x)=0\text{ on }\partial \Omega .  \label{dir2}
\end{equation}%
Using the Brezis-Oswald inequality \cite[Theorems 1 and 2]{BO}, we derive 
\begin{equation}
\underline{u}(x)\leq u^{0}(x)\leq u^{1}(x)\leq u^{2}(x)\leq \cdots \leq 
\overline{u}(x),\quad \forall \,x\in \overline{\Omega }.  \label{ineq}
\end{equation}%
Because the sequence $\{u^{k}\}_{k\geq 0}$ is monotonically increasing and
bounded above (by $\overline{u}(x)$), it follows from the monotone
convergence theorem that the pointwise limit%
\begin{equation*}
u(x)=\lim_{k\rightarrow \infty }u^{k}(x),
\end{equation*}%
exists for all $x\in \overline{\Omega }$. Standard regularity theory, which
can be found in \cite{G}, ensures that $u(x)\in C^{2}(\Omega )\cap C^{1}(%
\overline{\Omega })$ and, by construction, satisfies the required conditions
for problem \eqref{dir}. 
\begin{equation*}
\underline{u}(x)\leq u(x)\leq \overline{u}(x)\quad \forall \,x\in \overline{%
\Omega }.
\end{equation*}%
The uniqueness of the solution to problem \eqref{dir} follows again from the
results established by Brezis and Oswald \cite[Theorems 1 and 2]{BO} (see
also \cite{AP} or \cite{BK}).

Now, let us prove the validity of \eqref{2}. Assume that a solution of the
form 
\begin{equation}
v(x,t)=u^{\frac{1}{\alpha }}(x)\,S(t).  \label{sep}
\end{equation}%
Then we have%
\begin{equation*}
v^{\alpha }(x,t)=u(x)\,S^{\alpha }(t).
\end{equation*}%
Substitute (\ref{sep}) into (\ref{nd}) to obtain: 
\begin{equation}
u^{\frac{1}{\alpha }}(x)\,S^{\prime }(t)-S^{\alpha }(t)\,\Delta u\left(
x\right) =\mu \left( t\right) u^{\frac{1}{\alpha }}(x)\,S(t).  \label{ec}
\end{equation}%
Dividing (\ref{ec}) by $u^{\frac{1}{\alpha }}(x)\,$, we have 
\begin{equation}
S^{\prime }(t)-S^{\alpha }(t)\,\frac{\Delta u\left( x\right) }{u^{\frac{1}{%
\alpha }}(x)}=\mu \left( t\right) S\left( t\right) ,  \label{div}
\end{equation}%
or, equivalently%
\begin{equation}
S^{\prime }(t)+S^{\alpha }(t)\,=\mu \left( t\right) S\left( t\right) ,
\label{ber}
\end{equation}%
a Bernoulli-type equation. Multiplying both sides of (\ref{ber}) by $%
S^{-\alpha }(t)$ yields 
\begin{equation*}
S^{\prime }(t)S^{-\alpha }(t)\,=\mu \left( t\right) S^{1-\alpha }\left(
t\right) -1.
\end{equation*}%
We now solve this equation by the standard substitution. Define 
\begin{equation}
z(t)=S(t)^{\,1-\alpha }.  \label{sub}
\end{equation}%
Then, by differentiating in (\ref{sub}), we have 
\begin{equation*}
z^{\prime }(t)=\left( 1-\alpha \right) S^{-\alpha }(t)S^{\prime }\left(
t\right) \Longrightarrow S^{-\alpha }(t)S^{\prime }\left( t\right) =\frac{%
z^{\prime }(t)}{1-\alpha }.
\end{equation*}%
Substitute the expression for $S^{-\alpha }(t)S^{\prime }\left( t\right) $
into the above: 
\begin{equation*}
\frac{z^{\prime }(t)}{1-\alpha }=\mu \left( t\right) z(t)-1.
\end{equation*}%
This simplifies to 
\begin{equation}
z^{\prime }(t)=\left( 1-\alpha \right) \mu \left( t\right) z(t)+\alpha -1.
\label{lin1}
\end{equation}%
\bigskip Multiplying both sides of (\ref{lin1}) by $e^{-\left( 1-\alpha
\right) \int_{0}^{t}\mu \left( s\right) ds}$ yields%
\begin{equation*}
\left( z\left( t\right) e^{-\left( 1-\alpha \right) \int_{0}^{t}\mu \left(
s\right) ds}\right) ^{\prime }=\left( \alpha -1\right) e^{-\left( 1-\alpha
\right) \int_{0}^{t}\mu \left( s\right) ds},
\end{equation*}%
from where%
\begin{equation*}
z\left( t\right) =e^{\left( 1-\alpha \right) \int_{0}^{t}\mu \left( s\right)
ds}\int \left( \alpha -1\right) e^{-\left( 1-\alpha \right) \int_{0}^{t}\mu
\left( s\right) ds}dt.
\end{equation*}%
But recall that $S(t)^{\,1-\alpha }=z(t)$; hence, 
\begin{equation}
S\left( t\right) =\left[ e^{\left( 1-\alpha \right) \int_{0}^{t}\mu \left(
s\right) ds}\left( \gamma ^{1-\alpha }+(\alpha -1)\int_{0}^{t}e^{-\left(
1-\alpha \right) \int_{0}^{\tau }\mu \left( s\right) ds}d\tau \right) \right]
^{\frac{1}{1-\alpha }}.  \label{ff}
\end{equation}%
Replacing (\ref{ff}) in (\ref{sep}) we obtain the form (\ref{2}). This
concludes the proof of Theorem \ref{th1}.
\end{proof}

\begin{remark}
The Python script implementing the construction of the separable solution 
\begin{equation*}
v(x,t)=S(t)\cdot u^{1/\alpha }(x),
\end{equation*}%
as established in Theorem~\ref{th1}, is publicly available at \cite{TH1}.
\end{remark}

\begin{remark}
\label{rem-struct}The proof of Theorem~\ref{th1} reveals the essential
mechanism behind the Bernoulli localization principle: the substitution of
the product ansatz \eqref{sep} into the parabolic PDE \eqref{nd}
automatically produces the Bernoulli ODE \eqref{ber}, because the spatial
dependence cancels exactly when $u$ solves the elliptic problem \eqref{dir}.
This cancellation is not an accident; it is a structural feature of the
power-law nonlinearity $v^{\alpha}$, which ensures that $\Delta
v^{\alpha}=S^{\alpha}\Delta u$ is proportional to $v$ through the elliptic
equation. This structural observation is the conceptual heart of the
localization principle.
\end{remark}

We then establish a proof of the comparison principle.

\begin{proof}[Proof of Theorem \protect\ref{th-comp}]
Let 
\begin{equation*}
\mathcal{L}_{\mu }[u]:=\frac{\partial u}{\partial t}-\Delta u^{\alpha }-\mu
(t)u
\end{equation*}
be the parabolic operator associated with \eqref{nd}. By hypothesis, $v$ and 
$w$ are unique classical solutions satisfying 
\begin{equation*}
\mathcal{L}_{\mu _{1}}[v]=0\text{ and }\mathcal{L}_{\mu _{2}}[w]=0\text{ in }%
\Omega \times (0,T).
\end{equation*}
We first establish that $v$ serves as a subsolution for the operator $%
\mathcal{L}_{\mu _{2}}$. Indeed, substituting $v$ into $\mathcal{L}_{\mu
_{2}}$ yields: 
\begin{equation}
\mathcal{L}_{\mu _{2}}[v]=\frac{\partial v}{\partial t}-\Delta v^{\alpha
}-\mu _{2}(t)v=\mu _{1}(t)v-\mu _{2}(t)v=(\mu _{1}(t)-\mu _{2}(t))v.
\end{equation}%
Because $\mu _{1}(t)\leq \mu _{2}(t)$ for all $t$ and $v(x,t)\geq 0$ in $%
\Omega \times (0,T)$, it follows that $(\mu _{1}(t)-\mu _{2}(t))v\leq 0$,
which implies: 
\begin{equation}
\frac{\partial v}{\partial t}-\Delta v^{\alpha }\leq \mu _{2}(t)v\quad \text{%
in }\Omega \times (0,T).
\end{equation}%
On the other hand, the solution $w$ satisfies: 
\begin{equation}
\frac{\partial w}{\partial t}-\Delta w^{\alpha }=\mu _{2}(t)w\quad \text{in }%
\Omega \times (0,T).
\end{equation}%
Furthermore, the initial and boundary conditions satisfy 
\begin{equation*}
v(x,0)=v_{0}(x)\leq w_{0}(x)=w(x,0)\text{ in }\Omega \text{,}
\end{equation*}%
and 
\begin{equation*}
v(x,t)=0=w(x,t)\text{ on }\partial \Omega \times (0,T).
\end{equation*}

To establish the ordering $v\leq w$, we define 
\begin{equation*}
z(x,t):=v(x,t)-w(x,t)
\end{equation*}
and note that $z$ satisfies: 
\begin{equation}
\frac{\partial z}{\partial t}-\Delta (v^{\alpha }-w^{\alpha })=\mu
_{1}(t)v-\mu _{2}(t)w\leq \mu _{2}(t)(v-w)=\mu _{2}(t)z.
\end{equation}%
Using the monotonicity of the mapping $s\mapsto s^{\alpha }$ for $\alpha >1$
and $s\geq 0$, we can express the difference $v^{\alpha }-w^{\alpha }$ as $%
c(x,t)(v-w)$ where 
\begin{equation*}
c(x,t)=\int_{0}^{1}\alpha \lbrack w+\theta (v-w)]^{\alpha -1}d\theta \geq 0.
\end{equation*}%
Substituting this into the inequality for $z$, we obtain: 
\begin{equation}
\frac{\partial z}{\partial t}-\Delta (c(x,t)z)-\mu _{2}(t)z\leq 0.
\end{equation}%
Given that $z(x,0)\leq 0$ and $z$ vanishes at the boundary, the application
of the maximum principle for porous-medium-type equations (see \cite[Chapter
3]{V} or \cite[Theorem 2.1]{PAO}) ensures that $z(x,t)\leq 0$ on $\overline{%
\Omega }\times \lbrack 0,T)$. Consequently, $v(x,t)\leq w(x,t)$ for all $%
(x,t)\in \overline{\Omega }\times \lbrack 0,T)$, completing the proof.
\end{proof}

We conclude this section by establishing the proof of the third main result.

\begin{proof}[Proof of Theorem \protect\ref{th2}]
First we compute the quantities appearing in the parabolic operator for the
functions $v_{\pm }$.

\noindent \textbf{Step 1: Structure of $v_{\gamma }$.} Let $u$ be any of $%
u_{-}$ or $u_{+}$, and define%
\begin{equation*}
v_{\gamma }(x,t):=S(t)\,u^{1/\alpha }(x).
\end{equation*}%
Then%
\begin{equation*}
v_{\gamma }^{\alpha }(x,t)=\bigl(S(t)\,u^{1/\alpha }(x)\bigr)^{\alpha
}=S(t)^{\alpha }\,u(x).
\end{equation*}%
Hence, using that $S$ depends only on $t$,%
\begin{equation*}
\Delta \bigl(v_{\gamma }^{\alpha }(x,t)\bigr)=\Delta \bigl(S(t)^{\alpha }u(x)%
\bigr)=S(t)^{\alpha }\,\Delta u(x).
\end{equation*}%
Moreover,%
\begin{equation*}
\frac{\partial v_{\gamma }}{\partial t}(x,t)=S^{\prime }\left( t\right)
u^{1/\alpha }(x).
\end{equation*}%
Therefore, the left-hand side of the parabolic operator applied to $%
v_{\gamma }$ is 
\begin{align*}
\frac{\partial v_{\gamma }}{\partial t}(x,t)-\Delta \bigl(v_{\gamma
}^{\alpha }(x,t)\bigr)-\mu (t)\,v_{\gamma }(x,t)& =S^{\prime }\left(
t\right) u^{1/\alpha }(x)-S(t)^{\alpha }\,\Delta u(x)-\mu
(t)\,S(t)\,u^{1/\alpha }(x) \\
& =\bigl(S^{\prime }(t)-\mu (t)S(t)\bigr)\,u^{1/\alpha }(x)-S(t)^{\alpha
}\,\Delta u(x).
\end{align*}%
\noindent \textbf{Step 2: Use of the elliptic inequalities.}

\emph{(a) Subsolution case.} Assume $u=u_{-}$ satisfies \eqref{ell-sub}, i.e.%
\begin{equation*}
-\Delta u_{-}\leq u_{-}^{1/\alpha }\quad \text{in }\Omega .
\end{equation*}%
This is equivalent to%
\begin{equation*}
\Delta u_{-}\geq -u_{-}^{1/\alpha }\quad \text{in }\Omega .
\end{equation*}%
Then, for $v_{-}(x,t)=S(t)\,u_{-}^{1/\alpha }(x)$, we have 
\begin{align*}
\frac{\partial v_{-}}{\partial t}-\Delta \bigl(v_{-}^{\alpha }\bigr)-\mu
(t)\,v_{-}& =\bigl(S^{\prime }(t)-\mu (t)S(t)\bigr)\,u_{-}^{1/\alpha
}-S(t)^{\alpha }\,\Delta u_{-} \\
& \leq \bigl(S^{\prime }(t)-\mu (t)S(t)\bigr)\,u_{-}^{1/\alpha
}+S(t)^{\alpha }\,u_{-}^{1/\alpha } \\
& =\bigl(-S^{\alpha }(t)\bigr)\,u_{-}^{1/\alpha }+S(t)^{\alpha
}\,u_{-}^{1/\alpha }=0,
\end{align*}%
where we have used \eqref{S-ODE}. Therefore,%
\begin{equation*}
\frac{\partial v_{-}}{\partial t}-\Delta \bigl(v_{-}^{\alpha }\bigr)-\mu
(t)\,v_{-}\leq 0\quad \text{in }\Omega \times (0,T),
\end{equation*}%
which shows that $v_{-}$ is a subsolution of \eqref{nd}. On the boundary,
because $u_{-}=0$ on $\partial \Omega $, we have 
\begin{equation*}
v_{-}(x,t)=S(t)\,0^{1/\alpha }=0
\end{equation*}%
on $\partial \Omega \times (0,T)$.

\emph{(b) Supersolution case.} Assume $u=u_{+}$ satisfies \eqref{ell-super},
that is%
\begin{equation*}
-\Delta u_{+}\geq u_{+}^{1/\alpha }\quad \text{in }\Omega ,
\end{equation*}%
equivalently%
\begin{equation*}
\Delta u_{+}\leq -u_{+}^{1/\alpha }\quad \text{in }\Omega .
\end{equation*}%
Then, for $v_{+}(x,t)=S(t)\,u_{+}^{1/\alpha }(x)$, we obtain 
\begin{align*}
\frac{\partial v_{+}}{\partial t}-\Delta \bigl(v_{+}^{\alpha }\bigr)-\mu
(t)\,v_{+}& =\bigl(S^{\prime }(t)-\mu (t)S(t)\bigr)\,u_{+}^{1/\alpha
}-S(t)^{\alpha }\,\Delta u_{+} \\
& \geq \bigl(S^{\prime }(t)-\mu (t)S(t)\bigr)\,u_{+}^{1/\alpha
}+S(t)^{\alpha }\,u_{+}^{1/\alpha } \\
& =\bigl(-S^{\alpha }(t)\bigr)\,u_{+}^{1/\alpha }+S(t)^{\alpha
}\,u_{+}^{1/\alpha }.
\end{align*}%
Again, by \eqref{S-ODE}, the factor in parentheses vanishes, so%
\begin{equation*}
\frac{\partial v_{+}}{\partial t}-\Delta \bigl(v_{+}^{\alpha }\bigr)-\mu
(t)\,v_{+}\geq 0\quad \text{in }\Omega \times (0,T),
\end{equation*}%
where $v_{+}$ is the supersolution to \eqref{nd}. The boundary condition $%
v_{+}=0$ on $\partial \Omega \times (0,T)$ follows from $u_{+}=0$ on $%
\partial \Omega $.

\noindent \textbf{Step 3: Ordering of the solution via comparison.}

Assume now that $v$ is a classical solution of \eqref{nd} with%
\begin{equation*}
v(x,0)=v_{0}(x)\quad \text{in }\Omega ,\qquad v(x,t)=0\quad \text{on }%
\partial \Omega \times (0,T),
\end{equation*}%
and that%
\begin{equation*}
v_{-}(x,0)\leq v_{0}(x)\leq v_{+}(x,0)\quad \text{for all }x\in \Omega .
\end{equation*}%
By construction,%
\begin{equation*}
v_{-}(x,t)=S(t)\,u_{-}^{1/\alpha }(x),\qquad
v_{+}(x,t)=S(t)\,u_{+}^{1/\alpha }(x),\quad (x,t)\in \overline{\Omega }%
\times \lbrack 0,T).
\end{equation*}%
Therefore the initial data are sandwiched between the sub- and
supersolutions at time $t=0$. Moreover, for all $x\in \overline{\Omega }$
yields and $t\in \lbrack 0,T)$ we have 
\begin{equation*}
v_{-}(x,t)\leq v_{+}(x,t).
\end{equation*}%
To complete this argument, we invoke classical subsolution--supersolution
theory for nonlinear parabolic equations. We proceeded in several steps.

\noindent \textbf{Step 3.1: Initial ordering.} By hypothesis, the initial
data $v_{0}$ satisfies 
\begin{equation*}
v_{-}(x,0)\leq v_{0}(x)\leq v_{+}(x,0)\quad \text{for all }x\in \overline{%
\Omega }.
\end{equation*}%
Since%
\begin{equation*}
v_{-}(x,0)=\gamma \,u_{-}^{1/\alpha }(x)\text{ and }v_{+}(x,0)=\gamma
\,u_{+}^{1/\alpha }(x),
\end{equation*}%
and because $u_{-}\leq u_{+}$ in $\overline{\Omega }$, this ordering is
consistent with the assumptions.

\noindent\textbf{Step 3.2: Boundary compatibility.} On the boundary $%
\partial\Omega\times(0,T)$, both $v_{-}$ and $v_{+}$ vanish: 
\begin{equation*}
v_{-}(x,t)=S(t)\,u_{-}^{1/\alpha}(x)=0,\qquad
v_{+}(x,t)=S(t)\,u_{+}^{1/\alpha}(x)=0,
\end{equation*}%
since $u_{-}=u_{+}=0$ on $\partial\Omega$. This ensures that both the
subsolution and supersolution are compatible with the homogeneous Dirichlet
boundary condition \eqref{bnd1}.

\noindent \textbf{Step 3.3: Application of comparison principles.} The
theory developed by Pao \cite{PAO} (see also \cite{GM} for related results)
establishes that, under suitable regularity assumptions, if $v_{-}$ and $%
v_{+}$ are respectively an ordered subsolution and supersolution of the
parabolic problem \eqref{nd} with boundary conditions \eqref{bnd1}, and if
they satisfy the initial ordering at $t=0$, then there exists a unique
classical solution 
\begin{equation*}
v\in C^{2,1}(\Omega \times (0,T))\cap C(\overline{\Omega }\times \lbrack
0,T)),
\end{equation*}
of the problem \eqref{nd}--\eqref{bnd1}--\eqref{bnd2} such that 
\begin{equation}
v_{-}(x,t)\leq v(x,t)\leq v_{+}(x,t)\quad \text{for all }(x,t)\in \overline{%
\Omega }\times \lbrack 0,T).  \label{sandwich}
\end{equation}

\noindent \textbf{Step 3.4: Verification of the regularity hypotheses.} The
regularity assumptions required by the theory are satisfied:

\begin{itemize}
\item[(i)] The domain $\Omega$ is of class $C^{3}$ with compact closure;

\item[(ii)] The functions $u_{-},u_{+}\in C^{2}(\Omega)\cap C(\overline{%
\Omega})$ are strictly positive in $\Omega$ and vanish on $\partial\Omega$;

\item[(iii)] The function $S\in C^{1}([0,T])$ is strictly positive (since $%
\gamma>0$ and the ODE \eqref{S-ODE} preserves positivity);

\item[(iv)] The initial data $v_{0}\in C(\overline{\Omega})$ is sandwiched
between $v_{-}(\cdot,0)$ and $v_{+}(\cdot,0)$.
\end{itemize}

These conditions ensure that the nonlinear parabolic operator 
\begin{equation*}
\mathcal{L}[v]:=\frac{\partial v}{\partial t}-\Delta (v^{\alpha })-\mu (t)v
\end{equation*}%
is well defined and the subsolution--supersolution method applies.

\noindent\textbf{Step 3.4.1: Monotone iterative construction.} To make the
existence proof constructive and suitable for numerical implementation, we
describe the monotone iteration scheme following Pao \cite{PAO}. This
provides an algorithm for approximating the solution $v$ via sequences that
converge monotonically from below and above.

\emph{(a) Choice of monotonicity constant.} We first choose a function $%
c(x,t)\geq 0$ that is sufficiently large to ensure monotonicity.
Specifically, we consider the nonlinearity in \eqref{nd}, as follows: 
\begin{equation*}
f(v,t):=\mu (t)v.
\end{equation*}%
Since $\frac{\partial f}{\partial v}=\mu (t)$, we can take 
\begin{equation}
c(t):=\max \{0,\,\sup_{s\in \lbrack 0,t]}|\mu (s)|\},  \label{c-choice}
\end{equation}%
where clearly in the case $\mu (s)>0$ we obtain $c(t)=\sup_{s\in \lbrack
0,t]}\mu (s)$. This ensures that the auxiliary function 
\begin{equation*}
F(v,t):=c(t)\,v+f(v,t)=\bigl(c(t)+\mu (t)\bigr)v
\end{equation*}%
is nondecreasing in $v$ for $v\in \lbrack v_{-},v_{+}]$.

\emph{(b) Iterative scheme.} We define the two sequences $%
\{v_{-}^{(m)}\}_{m=0}^{\infty }$ and $\{v_{+}^{(m)}\}_{m=0}^{\infty }$ as
follows. Initialize with 
\begin{equation}
v_{-}^{(0)}(x,t):=v_{-}(x,t),\qquad v_{+}^{(0)}(x,t):=v_{+}(x,t).
\label{init-seq}
\end{equation}%
For $m\geq 1$, construct $v_{-}^{(m)}$ and $v_{+}^{(m)}$ by solving the
linearized parabolic problem 
\begin{equation}
\begin{aligned} \frac{\partial v^{(m)}}{\partial
t}-\Delta\bigl((v^{(m)})^{\alpha}\bigr)+c(t)\,v^{(m)}
&=c(t)\,v^{(m-1)}+\mu(t)\,v^{(m-1)}\quad\text{in }\Omega\times(0,T),\\
v^{(m)}(x,t)&=0\quad\text{on }\partial\Omega\times(0,T),\\
v^{(m)}(x,0)&=v_{0}(x)\quad\text{in }\Omega. \end{aligned}
\label{iter-scheme}
\end{equation}%
Here, the superscript $(m)$ denotes the iteration index, and $v^{(m-1)}$ is
the previous iterate.

\emph{(c) Monotonicity and convergence.} According to the comparison principle
for parabolic equations and the choice of $c(t)$, the sequences satisfy 
\begin{equation}
v_{-}=v_{-}^{(0)}\leq v_{-}^{(1)}\leq v_{-}^{(2)}\leq \cdots \leq
v_{-}^{(m)}\leq \cdots \leq v_{+}^{(m)}\leq \cdots \leq v_{+}^{(2)}\leq
v_{+}^{(1)}\leq v_{+}^{(0)}=v_{+}.  \label{monotone-order}
\end{equation}%
That is, the lower sequence $\{v_{-}^{(m)}\}$ monotonically increases and is
bounded above by $v_{+}$, whereas the upper sequence $\{v_{+}^{(m)}\}$
monotonically decreases and is bounded below by $v_{-}$. Standard parabolic
regularity estimates (see \cite{PAO}, Theorem 2.1) ensure that both
sequences converge uniformly on compact subsets of $\overline{\Omega }\times
\lbrack 0,T)$ to limits 
\begin{equation*}
\underline{v}(x,t):=\lim_{m\rightarrow \infty }v_{-}^{(m)}(x,t),\qquad 
\overline{v}(x,t):=\lim_{m\rightarrow \infty }v_{+}^{(m)}(x,t).
\end{equation*}%
Moreover, both $\underline{v}$ and $\overline{v}$ are classical solutions of %
\eqref{nd}--\eqref{bnd1}--\eqref{bnd2}, and they are the minimal and maximal
solutions, in the ordered interval $[v_{-},v_{+}]$.

\emph{(d) Uniqueness implies convergence to the same limit.} If the solution
is unique (as established in Step 3.5), then $\underline{v}\equiv\overline{v}%
\equiv v$, and both sequences converge to the same solution: 
\begin{equation}
\lim_{m\to\infty}v^{(m)}_{\pm}(x,t)=v(x,t)\quad\text{uniformly on }\overline{%
\Omega} \times[0,T).  \label{unique-conv}
\end{equation}

This constructive approach is particularly well-suited for numerical
implementation. At each iteration $m$, one solves a \emph{linear} parabolic
problem (because the right-hand side depends only on the previous iterate $%
v^{(m-1)}$), which can be discretized using standard finite difference or
finite element methods. Monotone convergence \eqref{monotone-order} provides
a natural stopping criterion based on the $L^{\infty }$ or $L^{2}$ norm of
the difference $\Vert v^{(m)}-v^{(m-1)}\Vert $.

\noindent \textbf{Step 3.5: Uniqueness. }The\textbf{\ }uniqueness of the
classical solution is derived from the monotonicity of the operator.
Specifically, if $v^{(1)}$ and $v^{(2)}$ are two solutions that lie in the
interval $[v_{-},v_{+}]$, then the difference $w:=v^{(1)}-v^{(2)}$ satisfies 
\begin{equation*}
\frac{\partial w}{\partial t}-\text{div}[\nabla (v^{(1)\alpha }-v^{(2)\alpha
})]-\mu (t)w=0.
\end{equation*}%
By mean-value expansions of the nonlinear term and the maximum principle for
parabolic equations, we conclude that $w\equiv 0$, yielding uniqueness.

\noindent \textbf{Step 3.6: Conclusion.} We have shown that $v_{-}$ and $%
v_{+}$ are respectively a subsolution and a supersolution of \eqref{nd} with
boundary condition \eqref{bnd1}, that they are ordered, and that the initial
data lie between them. Invoking the standard subsolution--supersolution
existence theorem (see \cite{PAO}, Theorem 1.3 and related results, or \cite%
{GM}, Section 3), we conclude that there exists a unique classical solution $%
v$ of problem \eqref{nd}--\eqref{bnd1}--\eqref{bnd2} satisfying the
pointwise estimate \eqref{sandwich}. This completes this proof.
\end{proof}

\begin{remark}
The Python code available at: \cite{TH2} implements Theorem~\ref{th2} using
the subsolution--supersolution framework. In the special case where the
solution admits a separable form 
\begin{equation*}
v(x,t)=S(t)\cdot u^{1/\alpha }(x),
\end{equation*}%
as in Theorem~\ref{th1}, this implementation confirms that both theorems
lead to identical numerical outcomes.
\end{remark}

\begin{remark}
The Python script available at: \cite{R} extends the implementation to illustrate Theorem~\ref{th2} under the general initial
condition proposed in Remark~\ref{rem}. In particular, the parabolic solver
is initialized with a weighted average of the sub- and supersolutions, 
\begin{equation*}
v_{0}(x)=\frac{7}{8}v_{-}(x,0)+\frac{1}{8}v_{+}(x,0),
\end{equation*}%
thereby demonstrating that the monotone iteration scheme robustly captures
the existence and uniqueness of the solutions for arbitrary initial data
within the ordered interval $[v_{-},v_{+}]$.
\end{remark}

In the following section, we present the derivation of the model from
fundamental conservation principles, connecting it to the classical
framework of Gurtin and MacCamy.


\section{Derivation of the Model via Conservation Laws \label{ef}}

The nonlinear diffusion equation \eqref{nd} can be derived from the
fundamental conservation law: 
\begin{equation}
\frac{\partial v}{\partial t}+\nabla \cdot \mathbf{J}=f(v,t),  \label{cons}
\end{equation}%
where $v(x,t)$ is the population density, $\mathbf{J}$ is the population
flux, and $f(v,t)$ is the net growth rate. The following constitutive
function is considered: 
\begin{equation*}
\varphi :[0,\infty )\rightarrow \lbrack 0,\infty ),
\end{equation*}%
which is differentiable in $v$ and satisfies the conditions: 
\begin{equation*}
\varphi ^{\prime }(0)=0,\quad \varphi ^{\prime }(v)>0\quad \text{for }v>0.
\end{equation*}%
In our model, the population flux $\mathbf{J}$ is governed by the potential $%
\varphi $, following a nonlinear Fick's law: 
\begin{equation}
\mathbf{J}=-\nabla \varphi (v)=-\varphi ^{\prime }(v)\nabla v.  \label{fick}
\end{equation}%
This indicates that dispersal is driven by the gradient of the potential $%
\varphi (v)$. Substituting \eqref{fick} into \eqref{cons} yields the
diffusion term $\nabla \cdot (\varphi ^{\prime }(v)\nabla v)$.

This flux-driven dynamic can be connected to a variational
principle. The Euler--Lagrange conditions associated with the functional are
inspired by the population flux $\mathbf{J}$, which suggest that the system
evolves to minimize the associated energy. We consider: 
\begin{equation*}
E(v)=\int_{\Omega }\left[ \frac{1}{2}\Bigl(\varphi ^{\prime }(v)\Bigr)^{2}+%
\frac{\mu \left( t\right) }{2}v^{2}\right] dx.
\end{equation*}%
Minimization of the energy $E(v)$, by using a flow in the positive direction
of the functional derivative, leads to: 
\begin{equation*}
\frac{\partial v}{\partial t}=\frac{\delta E}{\delta v}=\nabla \cdot \Bigl(%
\varphi ^{\prime }(v)\nabla v\Bigr)+\mu \left( t\right) v.
\end{equation*}%
Here, the variational derivation recovers the exact form of the flux term
derived from conservation laws. We observe that the diffusive term $\Delta
\varphi (v)$ can be expressed in the form: 
\begin{equation*}
\Delta \varphi (v)=\nabla \cdot \Bigl(\varphi ^{\prime }(v)\nabla v\Bigr).
\end{equation*}%
By choosing the specific function $\varphi (v)=v^{\alpha }$, the equation (%
\ref{nd}) can be obtained after rearrangement. The interpretations are as
follows: the diffusive term $\Delta v^{\alpha }$ describes the dispersal of
individuals in a habitat (representing the way in which the population
spreads out in space), whereas the term $\mu \left( t\right) v$ represents
the effect of population growth or mortality, modeling the birth and death
processes. The boundary conditions (\ref{bnd1}) and the initial condition (%
\ref{bnd2}) complement the model by ensuring that the initial population
distribution is predefined and that there is no flux of individuals across
the boundary of the domain.


\section{Cross-Disciplinary Relevance \label{cross}}

The Bernoulli localization principle developed in this paper is not confined
to a single application domain. In this section we show how Equation~%
\eqref{nd} and its structural decomposition via the Bernoulli ODE are
relevant to physics, biology, economics, and engineering. For each
discipline we identify the physical model, specify what the variables
represent, explain how the localization principle applies, and state a
concrete, verifiable prediction.

\subsection{Physics: thermal diffusion in porous media}

In the physics of porous media, the temperature $v(x,t)$ in a material with
temperature-dependent thermal conductivity $k(v)=\alpha v^{\alpha -1}$
satisfies the nonlinear heat equation%
\begin{equation}
\frac{\partial v}{\partial t}=\nabla \cdot (k(v)\nabla v)+Q(t)\,v=\Delta
v^{\alpha }+Q(t)\,v,  \label{thermal}
\end{equation}%
where $Q(t)$ is a time-dependent volumetric heat source (or sink). This is
precisely Equation~\eqref{nd} with $\mu (t)=Q(t)$. The power-law
dependence $k(v)=\alpha v^{\alpha -1}$ arises in radiative heat transfer ($%
\alpha =4$, Stefan--Boltzmann regime) and in the thermal conductivity of
certain polymers and ceramics \cite{Bar, DiBenedetto}.

\textbf{Verifiable prediction.} For a porous medium with thermal
conductivity $k(v)=\alpha v^{\alpha -1}$ (e.g., $\alpha =4$ for radiation-%
dominated transfer), initial temperature profile $v_{0}(x)$, and
time-dependent heat source $Q(t)=Q_{0}/(1+t)$, the Bernoulli localization
principle predicts that the temperature at any point $x$ and time $t$
satisfies%
\begin{equation*}
S(t)\,u_{-}^{1/\alpha }(x)\leq v(x,t)\leq S(t)\,u_{+}^{1/\alpha }(x),
\end{equation*}%
where $u_{-}$ and $u_{+}$ are the elliptic sub- and supersolutions on the
spatial domain, and $S(t)$ is computed from the closed-form expression %
\eqref{2} with $\mu (t)=Q_{0}/(1+t)$. These bounds can be tested against
numerical simulations or experimental measurements of heat flow in porous
ceramics.

\subsection{Biology: population dynamics and tumor growth}

Equation~\eqref{nd} was originally introduced by Gurtin and MacCamy \cite{GM}
to model the dispersal of biological populations with density-dependent
diffusion. In this context, $v(x,t)$ represents the population density, $%
\alpha >1$ captures the fact that organisms in crowded regions disperse more
rapidly (the ``slow diffusion'' regime, where the diffusion coefficient
vanishes when the density is zero, leading to compact support and finite-speed propagation), and $\mu (t)$ represents a time-varying net
reproduction rate.

A particularly important biological application is \textbf{tumor growth
modeling}. The density of tumor cells in a tissue region $\Omega $ evolves
according to a porous-medium-type equation, where the power-law diffusion $%
\Delta v^{\alpha }$ captures the mechanical pressure exerted by the growing
tumor mass, and $\mu (t)$ represents a time-varying proliferation rate that
may decrease due to nutrient depletion or therapeutic intervention \cite%
{Murray}.

\textbf{Verifiable prediction.} For a tumor with initial density profile $%
v_{0}(x)$ and proliferation rate $\mu (t)=\mu _{0}e^{-\beta t}$
(representing exponential decay due to chemotherapy), the maximum tumor cell
density at time $t$ is bounded above by $S(t)\max_{x}u_{+}^{1/\alpha }(x)$,
where $S(t)$ is the explicit Bernoulli solution~\eqref{2} with the given $%
\mu (t)$. This prediction can be tested against \emph{in vitro} spheroid
growth data.

\subsection{Economics and social sciences: diffusion of innovations}

The spread of a new technology, product, or idea through a population can be
modeled by nonlinear diffusion equations \cite{Rogers}. In this framework, $%
v(x,t)$ represents the density of adopters at location $x$ and time $t$,
the power-law diffusion $\Delta v^{\alpha }$ captures the nonlinear
``peer-influence'' effect (adoption accelerates when adopter density is
high), and $\mu (t)$ represents a time-varying exogenous stimulus (e.g.,
advertising, policy incentives, or media coverage).

The domain $\Omega $ represents a geographical region, and the Dirichlet
boundary condition $v=0$ on $\partial \Omega $ models the fact that adoption
does not penetrate beyond the market boundary. The Bernoulli localization
principle provides explicit bounds on the adoption density at any time and
location.

\textbf{Verifiable prediction.} For a technology diffusion process with
power-law peer influence ($\alpha =2$) and time-varying advertising
intensity $\mu (t)=\mu _{0}/(1+t)$ (diminishing returns), the total
adoption level $\int_{\Omega }v(x,t)\,dx$ is bounded between $%
S(t)\int_{\Omega }u_{-}^{1/2}(x)\,dx$ and $S(t)\int_{\Omega
}u_{+}^{1/2}(x)\,dx$, where $S(t)$ is given by \eqref{2}. These bounds can
be tested against empirical adoption curves for products such as
smartphones, electric vehicles, or renewable energy technologies.

\subsection{Engineering: groundwater flow and materials processing}

In hydrology, the hydraulic head $v(x,t)$ in an unconfined aquifer with
nonlinear permeability satisfies a porous medium equation of the form %
\eqref{nd}, where $\alpha $ depends on the soil type and $\mu (t)$
represents time-varying recharge or extraction rates \cite{Bear}. The
Bernoulli localization principle provides rigorous upper and lower bounds on
the hydraulic head, which are useful for predicting well yields, assessing
contamination risks, and designing groundwater management strategies.

In materials science, the sintering of powder compacts---a process by which
loose particles consolidate into a dense solid under heat---involves
density-dependent diffusion governed by equations of the form \eqref{nd},
where $v(x,t)$ represents the local density of the compact and $\mu (t)$
represents the time-dependent sintering temperature profile \cite{DiBenedetto%
}. The localization bounds predict the achievable density range at any stage
of the sintering process.

\textbf{Verifiable prediction.} For groundwater flow in a sandy aquifer ($%
\alpha \approx 2$) with seasonal recharge $\mu (t)=\mu _{0}(1+\cos (2\pi
t)) $, the hydraulic head at any monitoring well location $x$ is bounded by
the explicit Bernoulli barriers. These bounds can be compared against field
measurements from piezometric monitoring networks.


\section{Image Denoising Application and Comparison with Perona--Malik \label%
{a4}}

\subsection{Motivation: from biological diffusion to image processing}

From an applied perspective, one of the most striking consequences of the
framework developed in this paper is its direct applicability to digital
image processing. To our knowledge, this constitutes the \textbf{first
application} of a porous-medium-type (Gurtin--MacCamy) diffusion equation to
image denoising.

The connection between diffusion equations and image processing has been
recognized since the seminal work of Perona and Malik \cite{PM1, PM2}, who
proposed anisotropic diffusion as a method for edge-preserving image
smoothing. Their model, however, is based on a \emph{linear} diffusion
equation with a spatially varying, gradient-dependent diffusion coefficient.
The Gurtin--MacCamy model, by contrast, employs a genuinely \emph{nonlinear}
power-law diffusion mechanism $\Delta v^{\alpha }$ that has fundamentally
different mathematical properties: finite-speed propagation, degenerate
parabolicity, and compact support. These properties make it naturally suited
for image denoising, as the nonlinear diffusion automatically slows down
near edges (where the intensity is low after appropriate preprocessing) and
acts more aggressively in homogeneous regions.

\subsection{Mathematical formulation for image denoising}

We treat a grayscale image (or each color channel of a color image) as a
function $v:\Omega _{h}\times [0,T]\rightarrow [0,1]$, where $\Omega
_{h}\subset \mathbb{Z}^{2}$ is the discrete pixel grid. The denoising
process is governed by the discrete version of the Gurtin--MacCamy equation:%
\begin{equation}
v_{i,j}^{n+1}=v_{i,j}^{n}+\Delta t\left[ \Delta _{h}(v^{\alpha
})_{i,j}-\mu (n\Delta t)\,v_{i,j}^{n}\right] ,  \label{discrete-gm}
\end{equation}%
where $\Delta _{h}$ denotes the standard five-point discrete Laplacian%
\begin{equation}
\Delta _{h}(u)_{i,j}=u_{i+1,j}+u_{i-1,j}+u_{i,j+1}+u_{i,j-1}-4u_{i,j}.
\label{disc-lap}
\end{equation}%
In the denoising application, the term $-\mu (t)\,v$ acts as a
time-dependent damping that suppresses noise more aggressively at early
stages (when the noise level is high) and less aggressively as the image
converges to a clean state. The function $\mu (t)=0.3/(1+t)$ implements this
strategy: strong initial absorption that decays over time.

\subsection{Comparison with Perona--Malik anisotropic diffusion}

The Perona--Malik (PM) model \cite{PM1, PM2} is the standard benchmark for
edge-preserving image smoothing. It replaces the linear diffusion $\Delta v$
with the anisotropic operator%
\begin{equation}
\nabla \cdot \bigl(g(|\nabla v|)\,\nabla v\bigr),  \label{pm-eq}
\end{equation}%
where $g(s)=\exp (-s^{2}/K^{2})$ is an edge-stopping function. The PM model
reduces diffusion near edges (where $|\nabla v|$ is large) and maintains
full diffusion in flat regions.

The key difference is that the GM model achieves edge preservation through a
fundamentally different mechanism: the power-law nonlinearity $\Delta
v^{\alpha }$ with $\alpha >1$ automatically reduces the diffusion
coefficient when the intensity $v$ is small, without requiring explicit
gradient estimation. This makes the GM model more robust to noise in the
gradient computation, which is a well-known weakness of the PM approach at
high noise levels.

\subsection{Algorithmic description}

The complete denoising pipeline is summarized in Algorithm~\ref{alg:gm}.

\begin{algorithm}[H]
\caption{Gurtin--MacCamy Image Denoising}
\label{alg:gm}
\begin{algorithmic}[1]
\REQUIRE Clean image $I\in[0,1]^{M\times N\times 3}$, noise level $\sigma$,
exponent $\alpha$, time step $\Delta t$, number of steps $K$, growth
function $\mu(\cdot)$
\ENSURE Denoised image $I_{\mathrm{den}}\in[0,1]^{M\times N\times 3}$
\STATE Generate noisy image: $I_{\mathrm{noisy}} \leftarrow
\mathrm{clip}(I + \sigma\cdot\mathcal{N}(0,1),\;0,\;1)$
\FOR{each color channel $c\in\{R,G,B\}$}
    \STATE $v \leftarrow I_{\mathrm{noisy}}[:,:,c]$
    \FOR{$n = 0,1,\ldots,K-1$}
        \STATE $v_\alpha \leftarrow v^\alpha$ \hfill (elementwise power)
        \STATE $L \leftarrow \Delta_h(v_\alpha)$ \hfill (discrete Laplacian
        via \eqref{disc-lap})
        \STATE $v \leftarrow v + \Delta t\cdot(L - \mu(n\Delta t)\cdot v)$
        \STATE $v \leftarrow \mathrm{clip}(v,\;0,\;1)$
    \ENDFOR
    \STATE $I_{\mathrm{den}}[:,:,c] \leftarrow v$
\ENDFOR
\STATE Apply postprocessing (contrast stretch + gamma correction) to
$I_{\mathrm{den}}$
\RETURN $I_{\mathrm{den}}$
\end{algorithmic}
\end{algorithm}

\subsection{Python implementation: transparent explanation}

We now present the complete Python implementation with a detailed
explanation of each component. The goal is full transparency: every design
choice is justified mathematically or algorithmically, so that the reader
can reproduce, modify, and extend the code.

\subsubsection{Imports and parameters}

\begin{lstlisting}[caption={Imports and model parameters.}]
import numpy as np
import imageio.v2 as imageio
import matplotlib.pyplot as plt
from skimage.metrics import structural_similarity as ssim

alpha_gm = 4
dt_gm = 0.0118
steps_gm = 400

dt_pm = 0.2
steps_pm = 40
K_pm = 0.1

gamma_correction = 1.05
stretch_strength = 0.9
brightness_boost = 1.0

sigma = 0.18

def mu(t):
    return 0.3 / (1 + t)
\end{lstlisting}

\noindent\textbf{Explanation.} The code imports four libraries: \texttt{%
numpy} for array operations, \texttt{imageio} for reading image files, 
\texttt{matplotlib} for visualization, and \texttt{skimage.metrics} for the
SSIM quality metric.

The \emph{Gurtin--MacCamy parameters} are: $\alpha =4$ (the diffusion
exponent, corresponding to the radiation-dominated regime), $\Delta t=0.0118$
(the Euler time step, chosen small enough for stability of the explicit
scheme), and $K=400$ (the number of diffusion steps). The \emph{Perona--Malik
parameters} are: $\Delta t=0.2$, $K=40$, and $K_{\mathrm{pm}}=0.1$ (the
edge-stopping threshold). The postprocessing parameters control contrast
recovery after denoising (see below).

The noise level $\sigma =0.18$ specifies the standard deviation of the
additive Gaussian noise. The function $\mu (t)=0.3/(1+t)$ implements the
time-dependent growth rate from the theoretical framework: it provides
strong damping at early times (when noise is dominant) and progressively
weakens as the image approaches a clean state.

\subsubsection{Image quality metrics}

\begin{lstlisting}[caption={Quality metrics: MSE, PSNR, and SSIM.}]
def compute_mse(a, b):
    return np.mean((a - b)**2)

def compute_psnr(a, b):
    mse = compute_mse(a, b)
    if mse == 0:
        return float("inf")
    return 20 * np.log10(1.0 / np.sqrt(mse))

def compute_ssim_color(a, b):
    return np.mean([
        ssim(a[:,:,c], b[:,:,c], data_range=1.0)
        for c in range(3)
    ])
\end{lstlisting}

\noindent\textbf{Explanation.} Three standard image quality metrics are
implemented.

\begin{itemize}
\item \textbf{Mean Squared Error (MSE):} $\mathrm{MSE}=\frac{1}{MN}%
\sum_{i,j}(a_{i,j}-b_{i,j})^{2}$, measuring the average squared pixel
difference. Lower values indicate better reconstruction.

\item \textbf{Peak Signal-to-Noise Ratio (PSNR):} $\mathrm{PSNR}=20\log
_{10}(1/\sqrt{\mathrm{MSE}})$ dB, measuring the ratio of peak signal to
noise on a logarithmic scale. Higher values indicate better quality; a
difference of 3\,dB corresponds to a factor-of-two improvement in MSE.

\item \textbf{Structural Similarity Index (SSIM):} A perceptual metric that
compares luminance, contrast, and structural information between two images.
Values range from $-1$ to $1$, with $1$ indicating perfect similarity. For
color images, the SSIM is computed per channel and averaged.
\end{itemize}

\subsubsection{Postprocessing}

\begin{lstlisting}[caption={Postprocessing: contrast stretching and gamma
correction.}]
def postprocess(u):
    v = u.copy()
    v_min = v.min()
    v_max = v.max()
    if v_max > v_min:
        v = (v - v_min) / (v_max - v_min)
    v = stretch_strength * v + (1 - stretch_strength) * u
    v = np.power(v, gamma_correction)
    v *= brightness_boost
    return np.clip(v, 0, 1)
\end{lstlisting}

\noindent\textbf{Explanation.} After the diffusion process, the image
intensities may have shifted and compressed. The postprocessing restores
visual quality through three operations:

\begin{enumerate}
\item \textbf{Contrast stretching:} The image is linearly mapped to the
full $[0,1]$ range by $v\mapsto (v-v_{\min })/(v_{\max }-v_{\min })$. This
compensates for the intensity reduction caused by the damping term $-\mu
(t)\,v$ in the diffusion equation.

\item \textbf{Blending:} The stretched image is blended with the original
diffused image using the parameter \texttt{stretch\_strength}$=0.9$:%
\begin{equation*}
v_{\mathrm{out}}=0.9\cdot v_{\mathrm{stretched}}+0.1\cdot v_{\mathrm{%
diffused}}.
\end{equation*}%
This prevents over-enhancement of low-contrast regions.

\item \textbf{Gamma correction:} A mild gamma correction ($\gamma =1.05$)
is applied to adjust the luminance curve, slightly darkening the midtones
for a more natural appearance.
\end{enumerate}

\subsubsection{The Gurtin--MacCamy diffusion kernel}

\begin{lstlisting}[caption={Gurtin--MacCamy nonlinear diffusion for a
single color channel.}]
def diffuse_channel_gm(u, alpha, dt, steps):
    u = u.copy()
    for n in range(steps):
        u_alpha = u**alpha
        lap = (
            np.roll(u_alpha, 1, axis=0) +
            np.roll(u_alpha, -1, axis=0) +
            np.roll(u_alpha, 1, axis=1) +
            np.roll(u_alpha, -1, axis=1) -
            4*u_alpha
        )
        u = u + dt * (lap - mu(n*dt)*u)
        u = np.clip(u, 0, 1)
    return u
\end{lstlisting}

\noindent\textbf{Explanation.} This function implements the core of the
Gurtin--MacCamy denoising algorithm, corresponding to the discrete evolution
equation~\eqref{discrete-gm}. At each time step $n$:

\begin{enumerate}
\item \textbf{Compute $v^{\alpha }$:} The elementwise power $u^{\alpha }$ is
computed. For $\alpha =4$, this strongly attenuates low intensities (noise)
while preserving high intensities (signal).

\item \textbf{Compute the discrete Laplacian:} The five-point stencil %
\eqref{disc-lap} is implemented via \texttt{np.roll}, which shifts the array
along each axis. The sum of four shifted copies minus $4$ times the original
gives the standard finite-difference approximation to $\Delta (v^{\alpha })$.
The use of \texttt{np.roll} implements periodic boundary conditions, which is
standard practice in image processing and avoids boundary artifacts.

\item \textbf{Euler update:} The forward Euler step%
\begin{equation*}
v^{n+1}=v^{n}+\Delta t\bigl[\Delta _{h}(v^{\alpha })-\mu (n\Delta
t)\,v^{n}\bigr]
\end{equation*}%
advances the solution by one time step. The damping term $-\mu (n\Delta
t)\,v^{n}$ implements the time-dependent absorption from the Gurtin--MacCamy
framework.

\item \textbf{Clip to $[0,1]$:} Pixel values are clamped to the valid range
to maintain physical meaningfulness and numerical stability.
\end{enumerate}

The mathematical connection to the theoretical framework is direct: this
function implements a forward Euler discretization of the PDE $\partial
_{t}v=\Delta v^{\alpha }-\mu (t)v$, which corresponds to equation~%
\eqref{nd} with the sign of the reaction term adapted for the absorption
regime appropriate to denoising.

\subsubsection{The Perona--Malik diffusion kernel}

\begin{lstlisting}[caption={Perona--Malik anisotropic diffusion for a
single color channel.}]
def perona_malik_channel(u, dt, steps, K):
    u = u.copy()
    for _ in range(steps):
        ux_f = np.roll(u, -1, axis=1) - u
        ux_b = u - np.roll(u, 1, axis=1)
        uy_f = np.roll(u, -1, axis=0) - u
        uy_b = u - np.roll(u, 1, axis=0)

        cxf = np.exp(-(ux_f**2) / (K**2))
        cxb = np.exp(-(ux_b**2) / (K**2))
        cyf = np.exp(-(uy_f**2) / (K**2))
        cyb = np.exp(-(uy_b**2) / (K**2))

        div = (cxf*ux_f - cxb*ux_b + cyf*uy_f - cyb*uy_b)
        u = u + dt * div
        u = np.clip(u, 0, 1)
    return u
\end{lstlisting}

\noindent\textbf{Explanation.} This function implements the classical
Perona--Malik anisotropic diffusion model \cite{PM1, PM2} for comparison.
At each time step:

\begin{enumerate}
\item \textbf{Compute finite differences:} Forward and backward differences
are computed along both spatial axes (horizontal: $x$, vertical: $y$),
giving four gradient estimates per pixel.

\item \textbf{Edge-stopping function:} The Perona--Malik conductivity $%
g(s)=\exp (-s^{2}/K^{2})$ is applied to each gradient estimate. When the
gradient is large (near an edge), $g\approx 0$ and diffusion is suppressed;
when the gradient is small (in a flat region), $g\approx 1$ and full
diffusion occurs.

\item \textbf{Divergence and update:} The discrete divergence of the flux $%
g\cdot \nabla v$ is computed and the Euler step is applied, as in the
continuous equation $\partial _{t}v=\nabla \cdot (g(|\nabla v|)\nabla v)$.
\end{enumerate}

The Perona--Malik model requires choosing the threshold $K$ carefully: too
small, and noise is preserved; too large, and edges are blurred. This
sensitivity is a recognized limitation of the PM approach, particularly at
high noise levels where gradient estimates are unreliable.

\subsubsection{The complete denoising pipeline}

\begin{lstlisting}[caption={Main pipeline: noise addition, denoising,
and metric computation.}]
img = imageio.imread("lenna.png").astype(np.float32) / 255.0
gauss = sigma * np.random.randn(*img.shape)
noisy_img = np.clip(img + gauss, 0, 1)

den_gm = np.zeros_like(img)
den_pm = np.zeros_like(img)

for c in range(3):
    den_gm[:,:,c] = diffuse_channel_gm(
        noisy_img[:,:,c], alpha_gm, dt_gm, steps_gm
    )
    den_pm[:,:,c] = perona_malik_channel(
        noisy_img[:,:,c], dt_pm, steps_pm, K_pm
    )

den_gm_pp = postprocess(den_gm)
den_pm_pp = postprocess(den_pm)

for label, im in [("Noisy", noisy_img),
                   ("Gurtin-MacCamy", den_gm_pp),
                   ("Perona-Malik", den_pm_pp)]:
    print(f"{label}: MSE={compute_mse(img, im):.6f}, "
          f"PSNR={compute_psnr(img, im):.2f} dB, "
          f"SSIM={compute_ssim_color(img, im):.4f}")
\end{lstlisting}

\noindent\textbf{Explanation.} The main pipeline operates in four stages:

\begin{enumerate}
\item \textbf{Image loading and normalization:} The test image (\texttt{%
lenna.png}) is loaded and normalized to $[0,1]$ by dividing by $255$.

\item \textbf{Noise addition:} Additive white Gaussian noise with standard
deviation $\sigma $ is generated and added to the clean image. The result is
clipped to $[0,1]$ to maintain valid pixel values.

\item \textbf{Channel-wise denoising:} Both the GM and PM denoising
algorithms are applied independently to each of the three color channels
(R, G, B). This channel-independent approach is standard in color image
processing and avoids introducing color artifacts.

\item \textbf{Postprocessing and evaluation:} The postprocessing function
restores contrast to the denoised images, and the three quality
metrics (MSE, PSNR, SSIM) are computed for each method by comparing against
the original clean image.
\end{enumerate}

\subsection{Results and comparison}

The performance of our model, along with the metrics for the noisy image and
the Perona--Malik results, is summarized in Table~\ref{tab:denoising} for two
different noise levels ($\sigma=0.1$ and $\sigma=0.18$).

\begin{table}[tbp]
\caption{Comparative performance metrics for image denoising on \texttt{%
lenna.png}.}
\label{tab:denoising}\centering
\begin{tabular}{llccc}
\toprule Noise Level & Metric & Noisy Image & Gurtin--MacCamy & Perona--Malik
\\ 
\midrule $\sigma = 0.1$ & MSE & 0.009564 & \textbf{0.002459} & 0.002529 \\ 
& PSNR (dB) & 20.19 & \textbf{26.09} & 25.97 \\ 
& SSIM & 0.2960 & \textbf{0.6583} & 0.6355 \\ 
\midrule $\sigma = 0.18$ & MSE & 0.028305 & \textbf{0.004723} & 0.015066 \\ 
& PSNR (dB) & 15.48 & \textbf{23.26} & 18.22 \\ 
& SSIM & 0.1533 & \textbf{0.5249} & 0.2226 \\ 
\bottomrule &  &  &  & 
\end{tabular}%
\end{table}

\subsection{Discussion of the denoising results}

The results in Table~\ref{tab:denoising} reveal several important
observations.

\textbf{Low noise ($\sigma =0.1$).} At moderate noise levels, both the
Gurtin--MacCamy (GM) and Perona--Malik (PM) models achieve substantial
denoising, with the GM model holding a slight advantage across all three
metrics: MSE improves from $0.009564$ (noisy) to $0.002459$ (GM) versus $%
0.002529$ (PM); PSNR increases from $20.19$\,dB to $26.09$\,dB (GM) versus 
$25.97$\,dB (PM); and SSIM increases from $0.2960$ to $0.6583$ (GM) versus $%
0.6355$ (PM). The advantage of the GM model, while modest, is consistent.

\textbf{High noise ($\sigma =0.18$).} At elevated noise levels, the
superiority of the GM model becomes dramatic. The GM model achieves PSNR $%
=23.26$\,dB and SSIM $=0.5249$, while the PM model achieves only PSNR $%
=18.22$\,dB and SSIM $=0.2226$---a difference of over $5$\,dB in PSNR and
more than double the SSIM. The PM model's performance degrades sharply
because its gradient-based edge-stopping function $g(|\nabla v|)$ becomes
unreliable when the noise level is comparable to the gradient magnitudes of
true edges. The GM model, by contrast, does not rely on gradient estimation;
its edge-preserving capability arises from the intrinsic nonlinearity of the
power-law diffusion, which is robust to additive noise.

\textbf{The role of the time-dependent growth rate.} The function $\mu
(t)=0.3/(1+t)$ plays a crucial role in the GM model's success. At early
times, $\mu (t)$ is relatively large, providing strong absorption that
suppresses noise aggressively. As time progresses, $\mu (t)$ decreases,
allowing the diffusion to refine structural details without over-smoothing.
This adaptive behavior, which emerges naturally from the theoretical
framework of Equation~\eqref{nd}, is a distinctive advantage of the GM
approach.

The Python code implementing this comparison is publicly available at \cite%
{ID} for full reproducibility.


\section{Conclusions and Outlook \label{cf}}

\subsection{Summary of contributions}

In this paper, we have introduced the \emph{Bernoulli localization
principle}, a genuinely new approach to the analysis of nonlinear diffusion
equations of porous medium type with time-dependent growth rates. The
principle provides a complete structural description of solutions by
explicitly decoupling temporal dynamics---governed by a Bernoulli-type ODE
with a closed-form solution---from spatial profiles determined by a
stationary sublinear elliptic equation. This decomposition changes the way
the mathematical community can view this class of partial differential
equations: rather than treating the parabolic PDE as an opaque object
accessible only through qualitative comparison principles and energy
estimates, the localization principle reveals its internal structure in
explicit, computable terms.

The main theoretical contributions are:

\begin{enumerate}
\item \textbf{Theorem \ref{th1}:} We proved the existence and uniqueness of
separable classical solutions of the form 
\begin{equation*}
v(x,t)=S(t)u^{1/\alpha }(x),
\end{equation*}%
where the spatial profile $u(x)$ satisfies a sublinear elliptic equation and
the temporal function $S(t)$ is governed by a Bernoulli-type ODE. The
separation-of-variables approach reduces the infinite-dimensional parabolic
PDE to a scalar ODE plus a stationary elliptic problem, yielding explicit
analytical expressions that are directly implementable.

\item \textbf{Theorem \ref{th-comp}:} We established a comparison principle
showing that if the initial population $v_{0}$ is smaller than $w_{0}$ and
the growth rate $\mu _{1}(t)$ is smaller than $\mu _{2}(t)$, then the
solution $v(x,t)$ remains bounded by $w(x,t)$ for all subsequent times.

\item \textbf{Theorem \ref{th2}:} We extended the existence and uniqueness
theory to the general non-separable case using the monotone iteration method
based on subsolution--supersolution techniques, where the sub- and
supersolutions are explicitly constructed from the Bernoulli ODE solution.
This provides concrete, quantitatively sharp localization bounds for every
solution whose initial data lie in the ordered interval $[v_{-},v_{+}]$.

\item \textbf{Image denoising:} We presented the first application of a
porous-medium-type diffusion equation to digital image processing, showing
that the Gurtin--MacCamy model with time-dependent growth rate achieves
denoising performance comparable to or exceeding the classical Perona--Malik
filter, particularly at high noise levels.
\end{enumerate}

\subsection{Impact and broader significance}

The Bernoulli localization principle changes the paradigm for analyzing
porous-medium-type diffusion equations in several ways:

\begin{itemize}
\item It provides a \textbf{structural decomposition} that was not
previously available, making explicit what was formerly implicit.

\item It converts \textbf{qualitative comparison principles} into \textbf{%
quantitative localization bounds} with explicitly computable barriers.

\item It reduces the \textbf{computational complexity} of solving the
parabolic PDE by decoupling the temporal and spatial components.

\item It enables \textbf{direct parameter sensitivity analysis}, since the
dependence on $\alpha $, $\gamma $, and $\mu (\cdot )$ is transparently
readable from the closed-form expressions.
\end{itemize}

The cross-disciplinary relevance of the framework has been demonstrated
through concrete models and verifiable predictions in physics (thermal
diffusion in porous media), biology (population dynamics, tumor growth),
economics (diffusion of innovations), and engineering (groundwater flow,
materials processing). In each case, the Bernoulli localization principle
provides quantitative predictions that can be tested against data.

\subsection{Future research directions}

Several directions for future work are natural:

\begin{enumerate}
\item \textbf{Higher-dimensional domains and complex geometries:} Extension
of the numerical implementation to three-dimensional domains with irregular
boundaries, using finite element discretizations of the elliptic problem.

\item \textbf{Spatially varying diffusion coefficients:} Generalization of
the localization principle to equations of the form $\partial
_{t}v=\nabla \cdot (a(x)\nabla v^{\alpha })+\mu (t)v$, where $a(x)$
models spatial heterogeneity.

\item \textbf{Nonlinear boundary conditions:} Extension to problems with
nonlinear flux conditions on the boundary, which arise in models of
biological membranes and thermal radiation.

\item \textbf{Stochastic perturbations:} Analysis of the robustness of the
localization principle under stochastic forcing, relevant to models with
environmental noise.

\item \textbf{Advanced image processing:} Exploration of the GM diffusion
model for other image processing tasks such as segmentation, inpainting, and
super-resolution, potentially in combination with machine learning
techniques.

\item \textbf{Experimental validation:} Comparison of the theoretical
predictions with experimental data from laboratory measurements of diffusion
in porous materials, biological growth assays, or groundwater monitoring
networks.
\end{enumerate}

\section*{Acknowledgments}

The author thanks the anonymous referees for their careful reading and
helpful suggestions that improved the presentation of this paper.

\section*{Funding}

This research received no external funding.

\section*{Data Availability}

The Python code for all numerical experiments is publicly available at the
URLs listed in the references. The test image used for the denoising
experiments (\texttt{lenna.png}) is a standard benchmark image widely
available in the image processing community.

\section*{Declarations}

Conflicts of Interests: The author declares no conflict of interest.



\begin{thebibliography}{99}
\bibitem{AP} D.~G. Aronson and L.~A. Peletier. \newblock Large time
behaviour of solutions of the porous medium equation in bounded domains. %
\newblock {\em J. Differ. Equ.}, 39(3):378--412, 1981.

\bibitem{Bar} G.~I. Barenblatt. \newblock {\em Scaling, Self-similarity, and
Intermediate Asymptotics}. \newblock Cambridge Texts in Applied Mathematics.
Cambridge University Press, Cambridge, 1996.

\bibitem{Bear} J.~Bear. \newblock {\em Dynamics of Fluids in Porous Media}.
\newblock Dover Publications, New York, 1988.

\bibitem{BK} H.~Brezis and S.~Kamin. \newblock Sublinear elliptic equations
in $\mathbb{R}^{N}$. \newblock {\em Manuscripta Math.}, 74:87--106, 1992.

\bibitem{BO} H.~Brezis and L.~Oswald. \newblock Remarks on sublinear
elliptic equations. \newblock {\em Nonlinear Anal.}, 10:55--64, 1986.

\bibitem{Covei2} D.-P.~Covei, T.A.~Parvu, and C.~Sterbeti. \newblock The
equilibrium solutions for a nonlinear separable population model. \newblock
\emph{Mathematics}, 12(2):273, 2024.

\bibitem{Covei1} D.-P.~Covei, T.A.~Parvu, and C.~Sterbeti. \newblock A
population model with pseudo exponential survival. 
\newblock {\em Carpathian
J. Math.}, 40(3):627--641, 2024.

\bibitem{TH1} D.-P.~Covei, %
\url{https://github.com/coveidragos/PorousMediumEquation/blob/main/Theorem_First.py}

\bibitem{TH2} D.-P.~Covei, %
\url{https://github.com/coveidragos/PorousMediumEquation/blob/main/Theorem_Second.py}

\bibitem{R} D.-P.~Covei, %
\url{https://github.com/coveidragos/PorousMediumEquation/blob/main/Theorem_Second.py}

\bibitem{ID} D.-P.~Covei, %
\url{https://github.com/coveidragos/PorousMediumEquation/blob/main/Gurtin_Perona_Image.py}

\bibitem{DiBenedetto} E.~DiBenedetto. 
\newblock {\em Degenerate Parabolic
Equations}. \newblock Universitext. Springer-Verlag, New York, 1993.

\bibitem{Fisher} R.~A. Fisher. \newblock The wave of advance of advantageous
genes. \newblock {\em Ann. Eugenics}, 7:355--369, 1937.

\bibitem{G} D.~Gilbarg and N.~Trudinger. 
\newblock {\em Elliptic Partial
Differential Equations of Second Order}. \newblock Springer-Verlag, Berlin,
2001. \newblock Reprint of the 1998 Edition.

\bibitem{H} Q. He, \newblock Porous medium type reaction-diffusion equation:
Large time behaviors and regularity of free boundary. 
\newblock {\em
Journal of Functional Analysis}, 287 (2024) 110643.

\bibitem{GM} M.~E. Gurtin and R.~C. MacCamy. \newblock On the diffusion of
biological populations. \newblock {\em Math. Biosci.}, 33:35--49, 1977.

\bibitem{K} J. Kinnunen, P. Lindqvist and T. Lukkari, \newblock Perron's
method for the porous medium equation.\newblock {J. Eur. Math. Soc.}, 18,
2953--2969.

\bibitem{Murray} J.~D. Murray. 
\newblock {\em Mathematical Biology I: An
Introduction}. \newblock Interdisciplinary Applied Mathematics, Vol.~17.
Springer-Verlag, New York, 3rd edition, 2002.

\bibitem{Okubo} A.~Okubo and S.~A. Levin. 
\newblock {\em Diffusion and
Ecological Problems: Modern Perspectives}. \newblock Interdisciplinary
Applied Mathematics, Vol.~14. Springer-Verlag, New York, 2nd edition, 2001.

\bibitem{PAO} C.V.~Pao and W.H.~Ruan. \newblock Positive solutions of
quasilinear parabolic systems with Dirichlet boundary condition. \newblock
\emph{J. Differential Equations}, 248:1175--1211, 2010.

\bibitem{PM1} P.~ Perona, and J.~Malik. Scale-space and edge detection using
anisotropic diffusion, \textit{Proceedings of the IEEE Computer Society
Workshop on Computer Vision}, 16--22, 1987.

\bibitem{PM2} P.~ Perona, and J.~Malik. Scale-space and edge detection using
anisotropic diffusion, \textit{IEEE Transactions on Pattern Analysis and
Machine Intelligence}, \textbf{12}(7), 629--639, 1990.

\bibitem{Rogers} E.~M. Rogers. \newblock {\em Diffusion of Innovations}.
\newblock Free Press, New York, 5th edition, 2003.

\bibitem{V} J.L.~V\'{a}zquez. 
\newblock {\em The Porous Medium Equation:
Mathematical Theory}. \newblock Oxford Mathematical Monographs. Oxford
University Press, Oxford, 2007.

\bibitem{V2} J.L.~V\'{a}zquez. \newblock Fundamental solution and long time
behaviour of the porous medium equation in hyperbolic space. 
\newblock {\em
J. Math. Pures Appl.}, 104:454--484, 2015.
\end{thebibliography}
\end{document}